\documentclass[11pt]{article}

\usepackage{amsmath, amssymb, amsthm}
\usepackage{tikz}
\usetikzlibrary{arrows,shapes,positioning}
\usetikzlibrary{decorations.markings}
\usepackage{hyperref}
\usepackage{mathtools}

\theoremstyle{plain}
\newtheorem{theorem}{Theorem}[section]
\newtheorem{corollary}[theorem]{Corollary}
\newtheorem{proposition}[theorem]{Proposition}
\newtheorem{lemma}[theorem]{Lemma}

\theoremstyle{definition}
\newtheorem{definition}[theorem]{Definition}

\newcommand{\seqnum}[1]{\href{http://oeis.org/#1}{\underline{#1}}}
\newcommand{\sunderb}[2]{
  \mathclap{\underbrace{\makebox[#1]{$\cdots$}}_{#2}}
}

\tikzstyle arrowstyle=[scale=1.5]
\tikzstyle directed=[postaction={decorate,decoration={markings,
    mark=at position .65 with {\arrow[arrowstyle]{latex}}}}]

\begin{document}

\title{On the unimodality of convolutions of sequences of binomial coefficients
}

\author{Tricia Muldoon Brown\\
Georgia Southern University\\
\href{mailto:tmbrown@georgiasouthern.edu}{\tt tmbrown@georgiasouthern.edu}
}

\date{}

\maketitle
\begin{abstract}
We provide necessary and sufficient conditions on the unimodality of a convolution of two sequences of binomial coefficients preceded by a finite number of ones.  These convolution sequences arise as as rank sequences of posets of vertex-induced subtrees for a particular class of trees.  The number of such trees whose poset of vertex-induced subgraphs containing the root is not rank unimodal is determined for a fixed number of vertices $i$.
\end{abstract}
  
\section{Introduction}\label{sec_intro}

Unimodality of a sequence is an often-studied property where we say a sequence $\{s\}_{i\geq 0}$ is \emph{unimodal} if for some $k\geq 0$, we have
\begin{equation*}
s_0  \leq \cdots \leq s_{k-1}  \leq s_k \geq s_{k+1} \geq s_{k+2} \geq \cdots.
\end{equation*}
A classic example of a unimodal sequence is the sequence of binomial coefficients $\{{n\choose i}\}_{i\geq 0}$.  The question of unimodality of a sequence is a classic combinatorial problem.  Stanley~\cite{Stanley} provided a toolbox of techniques for proving unimodality and log-concavity which was updated by Brenti~\cite{Brenti}.  Further, several researchers have previously studied unimodality of other sequences involving binomial coefficients, including Tanny and Zuker~\cite{Tanny_Zuker} who establish the log-concavity and hence the unimodality of the sequence $\{{n-r\choose r}\}_{r\geq 0}$.  Belbachir, Bencherif, and Szalay~\cite{Belbachir_Bencherif_Szalay} find similar results for the sequence $\{{n+k\choose \beta k}\}_{k\geq 0}$ for some natural number $\beta \geq 2$.  Their conjecture on unimodality along a ray of Pascal's triangle is proven by Su and Wang~\cite{Su_Wang}.  Using the reflection principle, Sagan~\cite{Sagan} shows the unimodality of a sequence of products of binomial coefficients $\{{n\choose \ell -i}{n\choose i}\}_{0\leq i \leq \ell}$ for any $n$ and $\ell$, as well as the sequence of differences of products $\{{n\choose \ell -k}{n\choose k} - {n\choose \ell -k-1}{n\choose k-1}\}_{k\geq 0}$.  In particular, here we are interested in the unimodality of the convolution of two sequences of binomial coefficients preceded by a finite number of ones.  We define two sequences and their convolution as follows:

\begin{definition}
For integers $m,n,p,q\geq 0$, let
\begin{equation*}
\{s_i\}_{i\geq 1} = \left(\underbrace{1,1,\ldots,1}_p, {m\choose 0}, {m\choose 1}, \ldots, {m\choose m}, 0, 0, \ldots \right),
\end{equation*}
and 
\begin{equation*}
\{t_i\}_{i\geq 1} = \left(\underbrace{1,1,\ldots,1}_q, {n\choose 0}, {n\choose 1}, \ldots, {n\choose n}, 0, 0, \ldots \right).
\end{equation*}
Then define the sequence $\{r_i\}_{i\geq 1}$ to be the convolution of $\{s_i\}_{i\geq 0}$ and $\{t_i\}_{i\geq 1}$; that is, for $i\geq 1$,
\begin{equation*}
r_i = \sum_{j=1}^i s_j \cdot t_{i-j}.
\end{equation*}
\end{definition}

The main result provides necessary and sufficient conditions for $\{r_i\}_{i\geq 1}$ to be unimodal.  
\begin{theorem}\label{thm_seq_conditions}
For integers $m\geq n>0$ and $p,q\geq 0$, the sequence $\{r_i\}_{i\geq 1}$ unimodal if and only if at least one of the following conditions hold:
\begin{enumerate}
\item[i.] $m>q$,
\item[ii.] $n>p$,
\item[iii.] $m=n=2$, or
\item[iv.] $n=1$.
\end{enumerate}
\end{theorem}

As we shall see, the sequence $\{r_i\}_{i\geq 1}$ is motivated by graph theoretical results, so in Section~\ref{sec_tree} we formally define a class of trees composed of two broom graphs and the associated poset of connected, vertex-induced subgraphs, determining the rank sequence and sequence of first differences.  Section~\ref{sec_binomial} provides some intermediate results on sequences of sums and differences of binomial coefficients, while Section~\ref{sec_mainresult} proves the main result and consequently an analogous result for a poset of connected, vertex-induced subgraphs of two broom graphs.  We conclude in Section~\ref{sec_enumeration} with a closed formula to count the number of such trees with a fixed number of vertices whose poset is not rank unimodular.  We also note, throughout, $m,n,p$ and $q$ will be non-negative integers, $r$ will always refer to the root vertex of the tree in question, and all subtrees will be connected.

\section{Rank sequences and operations on rooted trees}\label{sec_tree}
Given a finite connected graph $G$, let $C(G)$ denote the poset of all connected, vertex-induced subgraphs of $G$ partially ordered by inclusion. This poset has been investigated by Leclerc~\cite{Leclerc} and Trotter and Moore~\cite{Trotter_Moore}, respectively, who give the dimension of this poset in the case of trees and graphs, respectively.  These subgraph posets appeared more recently in a problem solved by Steelman~\cite{Steelman} who gave conditions for when $C(G)$ is a lattice.  Further work has been done by K\'ezdy and Seif~\cite{Kezdy_Seif} on isomorphism conditions, and Vince and Wang~\cite{Vince_Wang} for an infinite set of non-Sperner subgraph posets.  Dropping the condition that the subgraphs must be vertex-induced, Smith and Tomon~\cite{Smith_Tomon} prove the poset of all connected subgraphs is Sperner.  Here we wish to consider a class of vertex-induced, connected subgraph posets.  Suppose $T$ is a rooted tree.  Let $C(T,r)$ be the poset of connected, vertex-induced subtrees of $G$ containing the root $r$ partially ordered by inclusion.  This poset and the more general poset were studied by Jacobson, K\'ezdy, and Seif~\cite{Jacobson_Kezdy_Seif} using the poset $C(T,r)$ to show $C(G)$ need not be Sperner.  

Each poset $C(T,r)$ has an associated sequence.
\begin{definition}
The \emph{rank sequence} $(r_i)_{i\geq 0}$ of a poset $P$ is given by the number of elements of rank $i$ in the poset.  These values $r_i$ are also called \emph{Whitney numbers} of the poset.  A poset is \emph{rank unimodal} if for some $0\leq k \leq n$, the rank sequence is unimodal.
\end{definition}

In particular, the Whitney number $r_i$ of $C(T,r)$ is given by the number of connected, vertex-induced subtrees of $T$ rooted at the root vertex with exactly $i$ vertices.

Examples of rank unimodal posets are found in many commonly studied posets, such as the Boolean lattice and its $q$-analogues, as well as the partition lattice.  The goal to characterize all trees whose poset of connected, vertex-induced subgraphs is unimodal was proposed by Jacobson, K\'ezdy, and Seif~\cite{Jacobson_Kezdy_Seif} .  They provide the only example of a tree whose vertex-induced subposet $C(T,r)$ is not unimodal when the number of vertices in the tree is less than or equal to 11, leading to questions on the prevalence of such non-unimodal posets.  This example, see Figure~\ref{fig_nuniexample}, is the motivating example for this work.

\begin{figure}
	\centering
	\begin{tikzpicture}
	\fill (3,4) circle (3pt);
	\fill (2,3) circle (3pt) ;
	\fill (1,2) circle (3pt) ;
	\fill (0,1) circle (3pt) ;
	\fill (1,1) circle (3pt) ;
	\fill (2,1) circle (3pt) ;
	\fill (4,3) circle (3pt) ;
	\fill (5,2) circle (3pt) ;
	\fill (6,1) circle (3pt) ;
	\fill (7,0) circle (3pt);
	\fill (5,0) circle (3pt);
	\draw[directed] (3,4)--(2,3);
	\draw[directed] (3,4)--(4,3);
	\draw[directed] (2,3)--(1,2);
	\draw[directed] (4,3)--(5,2);
	\draw[directed] (1,2)--(0,1);
	\draw[directed] (1,2)--(1,1);
	\draw[directed] (1,2)--(2,1);
	\draw[directed] (5,2)--(6,1);
	\draw[directed] (6,1)--(7,0);
	\draw[directed] (6,1)--(5,0);
	\end{tikzpicture}
\caption{The rooted tree on 11 vertices whose poset $C(T,r)$ has the non-unimodal rank sequence $(1,2,3,6,10,11,10,11,10,5,1)$.}
\label{fig_nuniexample}
\end{figure}
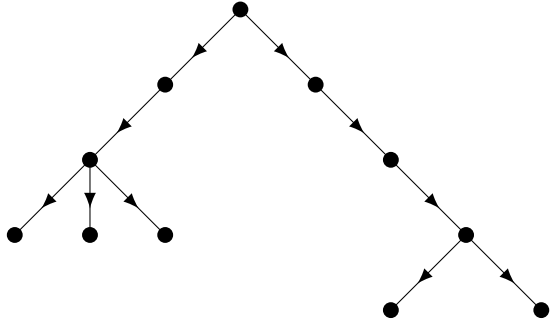

We begin by defining a simple tree which is illustrated in Fig~\ref{fig_broom}.

\begin{definition}
Given integers $m>0$ and $k\geq0$, the \emph{broom graph}, $B_{m,k}$, is a rooted, directed tree consisting of a path of length $k$ directed into a vertex with $m$ pendant vertices.  The root is the origin of the path and all edges are directed away from the root.
\end{definition}

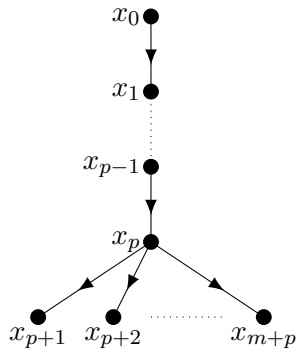
\begin{figure}
	\centering
	\begin{tikzpicture}
	\fill (1.5,4) circle (3pt) node[left]{$x_0$};
	\fill (1.5,3) circle (3pt) node[left]{$x_1$};
	\fill (1.5,2) circle (3pt) node[left]{$x_{p-1}$};
	\fill (1.5,1) circle (3pt) node[left]{$x_{p}$};
	\fill (0,0) circle (3pt) node[below]{$x_{p+1}$};
	\fill (1,0) circle (3pt) node[below]{$x_{p+2}$};
	\fill (3,0) circle (3pt) node[below]{$x_{m+p}$};
	\draw[directed] (1.5,4)--(1.5,3);
	\draw[style=dotted] (1.5,3)--(1.5,2);
	\draw[directed] (1.5,2)--(1.5,1);
	\draw[directed] (1.5,1)--(0,0);
	\draw[directed] (1.5,1)--(1,0);
	\draw[directed] (1.5,1)--(3,0);
	\draw[style=dotted] (1.5,0)--(2.5,0);
	\end{tikzpicture}
\caption{The broom graph $B_{m,p}$}
\label{fig_broom}
\end{figure}

It is not difficult to determine the rank sequence for the poset of connected, vertex-induced subtrees of a broom graph. 
\begin{lemma}
For $m>0$ and $p\geq 0$, the poset of connected, vertex-induced subtrees containing the root of a broom graph, $C(B_{m,p},r)$, is rank unimodal.
\end{lemma}
\begin{proof}
A broom without a handle, $B_{m,0}$, consists of a root vertex and $m$ pendant vertices.  The rank sequence of $C(B_{m,0},r)$ is given by the binomial coefficients, as the rooted subtrees with $i$ vertices are chosen by selecting $i-1$ edges from the set of $m$ pendant edges, that is, for $B_{m,0}$ we have $r_i = {m\choose i-1}$.  More generally if the broom has a handle of length $p$ and $m$ pendant vertices, because all subtrees in $B_{m,p}$ must be rooted, there is exactly one subtree with $i$ vertices for $1\leq i \leq p$.  All remaining subtrees in $B_{m,p}$ with $i>p$ vertices are subtrees of $B_{m,0}$ with $i-p$ vertices and a handle of $p$ vertices.  Thus the rank sequence is
\begin{equation*}
\left(\underbrace{1,1,\ldots,1}_p, {m\choose 0}, {m\choose 1}, \ldots, {m\choose m}\right).
\end{equation*}
As mentioned in Section~\ref{sec_intro}, it is well-known that binomial coefficients are unimodal, see Stanley~\cite{Stanley} for example, and hence $C(B_{m,p},r)$ is rank unimodal. 
\end{proof}

As the example provided by Jacobson, K\'ezdy, and Seif is two broom graphs, $B_{3,2}$ and $B_{2,3}$, whose root vertices have been identified, the question we wish to consider is for a tree $T$ composed of two broom graphs $B_{m,p}$ and $B_{n,q}$ whose root vertices have been identified, when is the subtree poset $C(T,r)$ rank unimodal?

We will classify a set of rooted trees by the unimodality of their rank sequences, so first we formalize an operation on rooted trees.

\begin{definition}
We \emph{merge} two rooted trees, $T_1$ and $T_2$, respectively, with root vertices, $v_1$ and $v_2$, respectively, by identifying the roots $v_1=v_2$.  We denote the new tree as $T_1 \cdot T_2$. 
\end{definition}

This operation on a pair of trees leads to an analogous operation on the rank sequences of their posets of vertex-induced subtrees containing the root.

\begin{lemma}\label{lemma_stretchmerge} 
If $(t_i)_{i\geq 1}$ is the rank sequence of $C(T,r_T)$ for a rooted tree $T$ and $\{s_i\}_{i\geq 1}$ is the rank sequence of $C(S,r_S)$ for a rooted tree $S$, then the rank sequence $(r_i)_{i\geq 1}$ of $C(T\cdot S,r)$ is given by the convolution $r_{i} =\sum_{j=0}^{i-1} t_{j+1} s_{i-j}$.
\end{lemma}

\begin{proof}
Let $T$ and $S$ be two rooted trees. Choosing a rooted subtree consisting of $i$ vertices from $T\cdot S$ is equivalent to choosing a rooted subtree of $j+1$ vertices from $T$ and a rooted subtree of $i-j$ vertices from $S$ for all $0 \leq j \leq i-1$.
\end{proof}

Thus we see the sequence $(r_i)_{i\geq i}$ defined in Section~\ref{sec_intro} is precisely the rank sequence of the poset $C(B_{m,p} \cdot B_{n,q}, r)$.  The following proposition provides a more specific description of $(r_i)_{i\geq 1}$ in terms of binomial coefficients.

\begin{proposition}\label{prop_ranksequence}
For integers $m,n >0$ and $p,q\geq 0$, let $(r_i)_{i\geq 1}$ be the rank sequence of $C(B_{m,p} \cdot B_{n,q},r)$, the poset of vertex-induced subtrees including the root of the merge of two broom graphs. If $p\leq q$, then
\begin{displaymath}
r_i = \begin{cases}
i, & \text{ for $i=1, 2, \ldots, p$;}\\
&\\
p  + \sum_{j=0}^{i-p-1} {m\choose j}, & \text{ for $i=p+1, p+2, \ldots, q$;} \\
&\\
\sum_{j=0}^{i-q-1} {n\choose j}+ (q+p-i) + \sum_{j=0}^{i-p-1} {m\choose j}, & \text{ for $i=q+1, q+2, \ldots, q+p$;}\\
&\\
\sum_{j=0}^{p-1} {n\choose i-q-1-j} + {m+n\choose i-q-p-1}+\sum_{j=0}^{q-1} {m\choose i-p-1-j},& \text{ for $i= q+p+1, \ldots, q+p+m+n+1$;}\\
&\\
0, & \text{ otherwise.}
\end{cases}
\end{displaymath}
\end{proposition}

\begin{proof}
By Proposition~\ref{lemma_stretchmerge}, for $1\leq i \leq p$ the rank sequence is the convolution of two sequences of ones, that is, 
\begin{equation*}
r_{i} = \sum_{j=0}^{i-1} 1\cdot 1=i.
\end{equation*}
Now, for $p<i\leq q$,
\begin{equation*}
r_{i}=\sum_{j=0}^{p-1} 1\cdot 1 + \sum_{j=p}^{i-1}{m\choose j-p}\cdot 1= p  + \sum_{j=0}^{i-p-1} {m\choose j}.
\end{equation*}
In the next case, when $q < i \leq q+p$,
\begin{eqnarray*}
r_{i} &=& \sum_{j=0}^{i-q-1} 1\cdot {n\choose j}+ \sum_{j=i-q}^{p-1} 1\cdot 1 + \sum_{j=p}^{i-1} {m\choose j-p} \cdot 1,\\
&=& \sum_{j=0}^{i-q-1} {n\choose j}+ (q+p-i) + \sum_{j=0}^{i-p-1} {m\choose j}.
\end{eqnarray*}
Finally for $q+p < i\leq q+p+m+n$, we have 
\begin{eqnarray*}
r_i&=& \sum_{j=0}^{p-1} 1\cdot {n\choose i-q-1-j} + \sum_{j=p}^{i-q-1} {m\choose j-p }{n\choose i-q-1-j}+\sum_{j=i-q}^{i-1} {m\choose j-p} \cdot 1\\
&=& \sum_{j=0}^{p-1} {n\choose i-q-1-j} + \sum_{j=0}^{i-q-p-1} {m\choose j }{n\choose i-q-p-1-j}+\sum_{j=0}^{q-1} {m\choose i-p-1-j}\\
&=& \sum_{j=0}^{p-1} {n\choose i-q-1-j} + {m+n\choose i-q-p-1}+\sum_{j=0}^{q-1} {m\choose i-p-1-j}.
\end{eqnarray*}
\end{proof}

Note, in Proposition~\ref{prop_ranksequence} we assumed that $p\leq q$, but should it be the case that $q\leq p$, a similar result may be obtained by replacing $p$ with $q$ and $q$ with $p$ throughout the statement of the result and the proof.  For simplicity throughout the rest of the paper, when referring to Proposition~\ref{prop_ranksequence} we will utilize whichever version is appropriate, whether $p\leq q$ or $q\leq p$.  Further results will only be explicitly proven for the case $p\leq q$, but it will be understood that we may apply a switch of variables to prove the result in the case $q \leq p$.

First differences of the rank sequence can be useful in determining unimodality, because if the sequence of first differences of a rank sequence changes sign at most one time, then the sequence is unimodal.  We will compute the first difference of the sequence $(r_i)_{i\geq 0}$, but first let us consider a few special cases.  

Suppose $q=p=0$.  Then $B_{m,0} \cdot B_{n,0}=B_{m+n,0}$ is the tree with root $r$ and exactly $m+n$ pendant vertices.  Therefore the rank sequence is precisely the sequence of binomial coefficients and thus is unimodal.  Next, suppose $q>p=0$.  In this case, applying Proposition~\ref{prop_ranksequence} to the rank sequence of $C(B_{m,0}\cdot B_{n,q}, r)$ gives the first differences
\begin{displaymath}
d_i =r_i -r_{i-1} =
\begin{cases}
{m\choose i-1}, & \text{ for $i= 2, 3,  \ldots, q$;} \\
&\\ 
{m+n\choose i-q-1} - {m+n\choose i-q-2}+{m\choose i-1} -{m\choose i-q-1},& \text{ for $i=q+1, \ldots, q+m+n$.}\\
\end{cases}
\end{displaymath}
In Section~\ref{sec_binomial}, we will apply Proposition~\ref{prop_c_unimodular} to show the unimodality of $(r_i)_{i\geq 1}$ in this case.

More generally we have the following corollary which follows directly from Proposition~\ref{prop_ranksequence} by subtraction.
\begin{corollary}\label{cor_difference}
For integers $m,n,p,q >0$, let $d_i = r_i - r_{i-1}$ be the first difference of the rank sequence, $(r_i)_{i\geq 1}$ of $C(B_{m,p}\cdot B_{n,q},r)$, the poset of connected, vertex-induced subtrees containing the root.  Without loss of generality assume $p\leq q$, then
\begin{displaymath}
d_i = 
\begin{cases}
1, & \text{ for $i=2, 3, \ldots, p$;}\\
&\\
{m\choose i-p-1}, & \text{ for $i=p+1, p+2, \ldots, q$;} \\
&\\
{n\choose i-q-1} -1 + {m\choose i-p-1},& \text{ for $i= q+1, q+2, \ldots, q+p$;} \\
&\\
{n\choose i-q-1}  -{n\choose i-q-p-1}+ {m+n\choose i-q-p-1}, & \text{ for $i=q+p+1, \ldots, q+p+m+n+1$;}\\
&\\
\indent - {m+n\choose i-q-p-2}+ {m\choose i-p-1} -{m\choose i-q-p-1} &\\
&\\
0, & \text{ otherwise.}
\end{cases}
\end{displaymath}
\end{corollary}

Now, for $i=1, \ldots,q$, because the first differences are non-negative, it is easily seen that the rank sequence is non-decreasing.  Further for $i> p+q+1+ \lfloor \frac{m+n}{2} \rfloor $, the rank sequence is decreasing.  To see this and to simplify notation, we rewrite $d_i$ for $i> p+q$ using the substitution $i = p+q+j+1$ to obtain the equation
\begin{equation}\label{eq_jsimplified}
d_i=d_{p+q+j+1} = {n\choose j+p}-{n\choose j}  + {m+n\choose j} - {m+n\choose j-1}+ {m\choose j+q} -{m\choose j}.
\end{equation}
We observe, for $j \geq \lfloor \frac{m+n}{2} \rfloor +1$ the differences ${m\choose j+p} - {m\choose j}$, ${m+n\choose j}-{m+n\choose j-1}$, and ${n\choose q+j} -{n\choose q}$ are all non-positive.  In fact, unless $j>m+n+1$ these differences cannot all be zero, because if ${m+n\choose j}-{m+n\choose j-1}=0$, we know $m+n$ is odd and $j=\lfloor \frac{m+n}{2}\rfloor +1$.  Because $m+n$ is odd, $m$ must be strictly greater than $n$ so ${m\choose \lfloor \frac{m+n}{2} \rfloor +1+q}-{m\choose \lfloor \frac{m+n}{2} \rfloor +1}\not=0$.  Thus the rank sequence is decreasing. 

Therefore the question of the rank unimodality of $(r_i)_{i\geq 1}$ is reduced to the the question of the unimodality of the terms $r_{q+1}, \ldots, r_{p+q+1+\lfloor \frac{m+n}{2} \rfloor}$.  Let us first consider a set of posets that are not rank unimodal.

\begin{proposition}\label{prop_notunimodular}
Given $m\geq n >0$ and $p,q > 0$, if $m\geq n \geq 3$ or $m>n=2$, the rank sequence $(r_i)_{i\geq 1}$ of the poset of vertex-induced subtrees containing the root, $C(B_{m,p}\cdot B_{n,q},r)$, is not unimodal if $q\geq m$ and $ p\geq n$.
\end{proposition}

\begin{proof}
Assume $q\geq m$ and $p\geq n$.  In Equation~\ref{eq_jsimplified}, set $j=1$.  Then
\begin{eqnarray*}
d_{q+p+2} &= &{n\choose p+1}-{n\choose 1}  + {m+n\choose 1} - {m+n\choose 0}+ {m\choose q+1} -{m\choose 1}\\
&=&{n\choose p+1} -n+n+m-1-m+ {m\choose q+1}\\
&=& {n\choose p+1} + {m\choose q+1}-1 =-1
\end{eqnarray*}
Thus $d_{q+p+2} <0$.  However, if $j=2$, then
\begin{eqnarray*}
d_{q+p+3} &= &{n\choose p+2}-{n\choose 2}  + {m+n\choose 2} - {m+n\choose 1}+ {m\choose q+2} -{m\choose 2}\\
&=&-{n\choose 2} + {m+n\choose 2} -(m+n)-{m\choose 2}\\
&=& -\frac{n(n-1)}{2} + \frac{(m+n)(m+n-1)}{2} -\frac{2(m+n)}{2} -\frac{m(m-1)}{2}\\
&=& \frac{1}{2}\left(-n^2+n +m^2+2mn+n^2-m-n -2m-2n -m^2+m\right)\\
&=&\frac{1}{2} (2mn -2m -2n)=mn-m-n \\
\end{eqnarray*}
When $m\geq n \geq 3$, it follows that $mn-m-n \geq 3m -m-n=2m-n\geq m> 0$, and when $m> n=2$ we have $mn-m-n = m-2>0$.  Thus, we see that $d_{q+p+3}>0$.  Easily $d_2=1>0$, so the sequence of first differences is positive, then negative, and then positive again.  Therefore the sequence is not unimodal. 
\end{proof}

We wish to show these bounds on $m, n, q,$ and $p$ are tight, but first we need some results on the unimodality of sums and differences of binomial coefficients.

\section{Unimodality of sequences of sums and differences of binomial coefficients}\label{sec_binomial}

The sequence of first differences $(d_i)_{i\geq 2}$ of the sequence $(r_i)_{i\geq 1}$ given in Section~\ref{sec_tree} is composed of sums and differences of binomial coefficients, so in this section we investigate some similar sequences.  In particular, the difference we are concerned with here is sequence \seqnum{A080232} in the OEIS~\cite{OEIS}, namely, the sequence $\big({m\choose j}-{m\choose j-1}\big)_{j \geq 0}$.  This sequence has known combinatorial interpretations enumerating ballot sequences as well as a subset of lattice paths from $(0,0)$ to $(j, m-j)$.  We are interested in the unimodality of the first half of this sequence.

\begin{lemma}\label{lem_s_unimodal}
Given $m\geq 0$, the sequence $s_{m,j}={m\choose j}-{m\choose j-1}$ is unimodal with respect to $j$ on the interval $0\leq j \leq \lfloor \frac{m}{2} \rfloor$.  
\end{lemma}

\begin{proof}
We proceed by induction on $m$.  Easily, if $m=0$, the sequence $s_{0,0}=1$ is unimodal.  Now for $m>0$, suppose the sequence $s_{m-1, 0}, s_{m-1,1}, \ldots, s_{m-1,\lfloor \frac{m-1}{2}\rfloor}$ is unimodal with a peak at $s_{m-1,k}$, that is
\begin{equation*}
s_{m-1, 0} \leq s_{m-1,1} \leq \cdots \leq s_{m-1,k} \geq s_{m-1,k+1} \geq \cdots \geq s_{m-1,\lfloor \frac{m-1}{2}\rfloor}.
\end{equation*}
As the sequence $s$ follows the recursion on the binomial coefficients, $s_{m,j} = s_{m-1,j}+s_{m-1,j-1}$, we see
\begin{align*}
s_{m-1, k-i} \geq s_{m-1,k-i-2} &\implies  s_{m,k-i} \geq s_{m,k-i-1}, \hbox{ for } 0\leq i \leq k\\
s_{m-1, k+i} \geq s_{m-1,k+i+2} &\implies  s_{m,k+i+1} \geq s_{m,k+i+2} \hbox{ for } 0 \leq i \leq \left\lfloor \frac{m-1}{2} \right\rfloor -k-2.
\end{align*}
To check the end condition, recall $\lfloor \frac{m-1}{2} \rfloor = \lfloor \frac{m}{2} \rfloor -1$ and $s_{m-1, \lfloor \frac{m-1}{2} \rfloor +1 }=0$ if $m-1$ is odd, so 
\begin{align*}
 s_{m-1, \lfloor \frac{m-1}{2} \rfloor -1} & \geq  s_{m-1, \lfloor \frac{m-1}{2} \rfloor+1}=0\\
s_{m-1, \lfloor \frac{m-1}{2} \rfloor} + s_{m-1, \lfloor \frac{m-1}{2} \rfloor -1} &\geq s_{m-1, \lfloor \frac{m-1}{2} \rfloor+1} + s_{m-1, \lfloor \frac{m-1}{2} \rfloor }\\
s_{m-1, \lfloor \frac{m}{2} \rfloor -1}+s_{m-1, \lfloor \frac{m}{2} \rfloor -2}  &\geq  s_{m-1, \lfloor \frac{m}{2} \rfloor}+s_{m-1, \lfloor \frac{m}{2} \rfloor-1}\\
s_{m, \lfloor \frac{m}{2} \rfloor -1} &\geq  s_{m, \lfloor \frac{m}{2} \rfloor}.\\
\end{align*}
If $m-1$ is even, $\lfloor \frac{m-1}{2} \rfloor = \lfloor \frac{m}{2} \rfloor $, so 
\begin{equation*}
s_{m-1, \lfloor \frac{m-1}{2} \rfloor -2} \geq  s_{m-1, \lfloor \frac{m-1}{2} \rfloor} \implies s_{m, \lfloor \frac{m}{2} \rfloor -1} \geq  s_{m, \lfloor \frac{m}{2} \rfloor}.
\end{equation*}
Thus the sequence increases from $s_{m,0}$ to $s_{m,k}$ and decreases from $s_{m,k+1}$ to $s_{m, \lfloor \frac{m}{2} \rfloor}$.  No matter the relationship between $s_{m,k}$ and $s_{m,k+1}$ the sequence is unimodal.
\end{proof}

Note the proof techniques in Lemma~\ref{lem_s_unimodal} may be generalized to other sequences found in a row of a triangular array which satisfies the binomial recurrence.  We state such a result.
\begin{proposition}\label{prop_c_unimodular}
For integers $m,n \geq 0$ and $q\geq 0$, the sequence \begin{equation*}c_{m,n,j}={m+n \choose j} - {m+n\choose j-1} +{m\choose j+q}-{m\choose j}\end{equation*} is unimodal with respect to $j$ on the interval $0\leq j \leq \lfloor \frac{m+n}{2} \rfloor$.
\end{proposition}

\begin{proof}
Start by fixing $m= 0$.  Then $c_{0,n,j}= {n \choose j} - {n\choose j-1} - {0\choose j}+{0\choose j+q}$.  Create a triangular array for $n\geq j \geq 0$.  If $q=0$, then $c_{0,n,j} = s_{n,j}$, the sequence found in Lemma~\ref{lem_s_unimodal}, and is therefore unimodal.  If $q>0$, then 
\begin{displaymath}
c_{0,n,j} = 
\begin{cases}
s_{n,j}, & \text{ if $j\not=0$;}\\
s_{n,0}-1, & \text{ if $j=0$.}\\
\end{cases}
\end{displaymath}
is also unimodal.

Now for a fixed $n$, create a triangular array for $c_{m,n,j}$ where $m>0$ and $j\geq 0$.  
\[
\begin{array}{cccccccccc}
&&&c_{0,n,0} & \cdots & c_{0,n,n} &&&&\\
&&c_{1,n,0}  &&  c_{1,n,1} & \cdots & c_{1,n,n+1} &&&\\
&c_{2,n,0} && c_{2,n,1} && c_{2, n,2} & \cdots & c_{2,n ,n+2} &&\\
c_{3,n,0} && c_{3,n,1} && c_{3,n,2} &&c_{3,n,3} &\cdots & c_{3,n,n+3} &\\
&&&&\cdots &&&&&
\end{array}
\]
The initial row of this array is unimodal as it appears in the array of $c_{0,n,j}$.  Further, the binomial recurrence, $c_{m, n, j} = c_{m-1,n,j}+c_{m-1,n, j-1}$ is satisfied, as the terms are sums and differences of binomial coefficients with respect to $m$ and $j$.  Repeating the proof techniques of Lemma~\ref{lem_s_unimodal}, we see that every row in a triangular array having the binomial recurrence and an initial unimodular row is also unimodal, and we have proven the claim.
\end{proof}

Recall the first difference of the rank sequence of $C(B_{m,0} \cdot B_{n,q}, r)$ given in Section~\ref{sec_tree}.  Because a unimodular sequence whose first term is non-negative may change signs at most once and because the sequence $c_{m,n,j}$ is non-negative when $j=0$, Proposition~\ref{prop_c_unimodular} and the substitution $j=i-q-1$, along with the discussion in Section~\ref{sec_tree}, imply the following corollary.
\begin{corollary}\label{cor_cases}
For $m,n,q>0$, the rank sequences $(r_i)_{i\geq 1}$ of the posets $C(B_{m,0}\cdot B_{n,0},r)$ and $C(B_{m,0}\cdot B_{n,q},r)$ are rank unimodal.
\end{corollary}

This result will be used to prove the necessary and sufficient conditions on the integers $m,n,p,q$ to determine rank unimodality of the poset $C(B_{m,p}\cdot B_{n,q},r)$.

\section{Proofs}\label{sec_mainresult}

Before proving Theorem~\ref{thm_seq_conditions}, we need to prove some intermediate results.
\begin{proposition}\label{prop_less_zero}
Given integers $m\geq n >0$, if ${m+n\choose j}-{m+n\choose j-1} - {m\choose j} \leq 0$ for some $0< j \leq \lfloor \frac{m+n}{2} \rfloor$, then $j\geq n$.
\end{proposition}

\begin{proof}
Suppose by assumption ${m\choose j}\geq{m+n\choose j} - {m+n\choose j-1}$.  This implies
\begin{align*}
\frac{m!}{j!(m-j)!} &\geq \frac{(m+n)!}{j!(m+n-j)!}\left(1-\frac{j}{m+n+1-j}\right)\\
\frac{m!}{j!(m-j)!} &\geq \frac{(m+n)!}{j!(m+n-j)!}\left(\frac{m+n+1-2j}{m+n+1-j}\right)\\
\frac{m+n+1-j}{m+n+1-2j} &\geq \frac{(m+n)(m+n-1)\cdots (m+1)}{ (m+n-j)(m+n-1-j)\cdots (m+1-j)}\\
1+\frac{j}{m+n+1-2j} &\geq \left(1+\frac{j}{m+n-j}\right) \left(1+\frac{j}{m+n-1-j}\right)\cdots \left(1+\frac{j}{m+1-j}\right)
\end{align*}

As each factor in the product is greater than one, we have
\begin{align*}
1+\frac{j}{m+n+1-2j} &\geq1+\frac{j}{m+1-j}\\
 m+1-j &\geq m+n+1-2j \\
 j&\geq  n.
\end{align*}
\end{proof}

Proposition~\ref{prop_less_zero} implies the following corollary we will need to prove our main result.

\begin{corollary}\label{cor_less_zero}
Given integers $m\geq n >0$ and $p,q\geq 0$, if 
\begin{equation*} {m+n\choose j}-{m+n\choose j-1} - {m\choose j} + {m\choose j+q}\leq0
\end{equation*}  for some $0< j \leq \lfloor \frac{m+n}{2} \rfloor$, then \begin{equation*} {m+n\choose j}-{m+n\choose j-1} - {m\choose j} + {m\choose j+q} - {n\choose j} + {n\choose j+q}\leq0 . \end{equation*}
\end{corollary}

The assumptions in Corollary~\ref{cor_less_zero} imply those of Proposition~\ref{prop_less_zero}, so $j\geq n$ and thus the quantity $-{n\choose j} +{n\choose j+q}  \leq 0$.

Now we consider the case that the sequence ${m+n\choose j}-{m+n\choose j-1}-{m\choose j}$ is greater than zero.

\begin{proposition}\label{prop_greater_zero}
Given integers $m\geq n>0$, if \begin{equation*}{m+n\choose j}-{m+n\choose j-1}-{m\choose j} >0\end{equation*} for some $1<j \leq \lfloor \frac{m+n}{2} \rfloor$, then \begin{equation*}{m+n\choose j}-{m+n\choose j-1}-{m\choose j} -{n\choose j}\geq 0.\end{equation*}
\end{proposition}

\begin{proof}
First, if $j>n$, then ${n\choose j}=0$ and the result holds, so assume $1< j \leq n$; that is, assume ${n\choose j} \not=0$ and ${m+n\choose j}-{m+n\choose j-1}-{m\choose j} \geq 0$.  We wish to show
\begin{equation*}
{m+n\choose j}-{m+n\choose j-1}-{m\choose j}-{n\choose j} \geq 0.
\end{equation*}
But dividing by ${n\choose j}$ and then adding 1 to both sides, the inequality can be rewritten as follows:
\begin{align*}
\frac{(m+n)\cdots (n+1)}{(m+n-j)\cdots (n+1-j)} \left(\frac{m+n+1-2j}{m+n+1-j} \right) -\frac{m(m-1)\cdots (n+1)}{(m-j)(m-1-j)\cdots (n+1-j)} &\geq1 \\
\frac{(m+n)}{n} \cdots \frac{(m+n+1-j)}{(n+1-j)} \left(\frac{m+n+1-2j}{m+n+1-j} \right) -\frac{m}{n}\cdot\frac{(m-1)}{(n-1)}\cdots \frac{(m+1-j)}{(n+1-j)} &\geq1\\
\frac{(m+n)(m+n-1)\cdots (m+n+2-j)(m+n+1-2j)-m(m-1)\cdots (m+1-j)}{n(n-1)\cdots (n+1-j)} &\geq  1
\end{align*}

To show this inequality holds, we induct on $j\geq 2$.
Suppose $j=2$.  Then we have the expression
\begin{equation*}
\frac{(m+n)(m+n-3)-m(m-1)}{n(n-1)}=\frac{2(m+n)}{n}-\frac{n+1}{n-1}.
\end{equation*}
Because $m+n\geq2n$, we have
\begin{equation*}
\frac{2(m+n)}{n}-\frac{n+1}{n-1} \geq 4 - \frac{n+1}{n-1}.
\end{equation*}
Then $\frac{n+1}{n-1}\leq 3$ for $n\geq 2$, so we have the desired result
\begin{equation*}
\frac{(m+n)(m+n-3)-m(m-1)}{n(n-1)}\geq 1.
\end{equation*}
Now by induction assume the inequality
\begin{equation}\label{ineq_j}
\frac{(m+n)\cdots (m+n+2-j)(m+n+1-2j)-m\cdots (m+2-j)(m+1-j)}{n(n-1) \cdots (n+1-j)} \geq  1\\
\end{equation}
holds for $j\geq 2$.
We wish to multiply by a term which preserves the inequality.
First, let us assume the strict inequality $n>j$; that is, $n\geq j+1$.  Then
\begin{align*}
m+n &\geq n+j +1\\
m+n-1-2j & \geq  n-j.
\end{align*}
Clearly $m+n+1-j \geq m+n+1-2j$, so the product
\begin{equation*}
\frac{(m+n+1-j)}{(m+n+1-2j)}\cdot \frac{(m+n-1-2j)}{(n-j)}\geq 1.
\end{equation*}
Thus multiplying the left-hand side of inequality~\ref{ineq_j} by this term maintains the inequality on the right-hand side.  Multiply the denominator by $(n-j)$ to obtain \begin{equation*}n(n-1) \cdots (n+1-j)(n-j),\end{equation*} as desired.  Multiply the numerator by the remaining factor, $\frac{(m+n+1-j)(m+n-1-2j)}{(m+n+1-2j)}$, to obtain 
\begin{equation*}
\resizebox{.98\hsize}{!}{$(m+n)\cdots (m+n+1-j)(m+n-1-2j)-m\cdots (m+1-j)\frac{(m+n+1-j)(m+n-1-2j)}{(m+n+1-2j)}$.}
\end{equation*}
The first summand is as desired for the inductive step, so consider the the second term.  We have
\begin{align*}
0 &\leq n-1 = 2n^2-n-1-2n(n-1) \leq (n-1)(2n+1) -2nj\\
&\leq  (n-1)(m+n+1)-2nj \\
&=  (m+n+1-j)(m+n-1-2j)-(m-j)(m+n+1-2j).
\end{align*}
Thus,
\begin{align*}
(m-j)(m+n+1-2j) & \leq  (m+n+1-j)(m+n-1-2j)\\
m-j & \leq \frac{(m+n+1-j)(m+n-1-2j)}{m+n+1-2j},
\end{align*}
which implies
\begin{equation*}
\resizebox{.98\hsize}{!}{$-m(m-1)\cdots (m+1-j)\frac{(m+n+1-j)(m+n-1-2j)}{(m+n+1-2j)} \leq - m(m-1)\cdots (m+1-j)(m-j)$.}
\end{equation*}
Thus for the inductive step, that is $j+1$, the following inequality holds:
\begin{equation*}
\frac{(m+n)\cdots (m+n+1-j)(m+n-1-2j)- m\cdots (m+1-j)(m-j) }{n(n-1) \cdots (n+1-j)(n-j)} \geq 1
\end{equation*}
We have one other case to consider.  Suppose $j=n$.
Then by assumption
\begin{equation*}
{m+n\choose n} - {m+n\choose n-1} - {m\choose n} >0,
\end{equation*}
therefore
\begin{equation*}
{m+n\choose n} - {m+n\choose n-1} - {m\choose n} - {n\choose n} = {m+n\choose n} - {m+n\choose n-1} - {m\choose n} - 1 \geq 0.
\end{equation*}
 \end{proof}

We may extend Proposition~\ref{prop_greater_zero} as follows:

\begin{corollary}\label{cor_greater_zero}
Given integers $m\geq n>0$ and $p,q\geq 0$, if \begin{equation*}{m+n\choose j}-{m+n\choose j-1}-{m\choose j} +{m\choose j+q}>0\end{equation*} for some $1<j \leq \lfloor\frac{m+n}{2}\rfloor$, then \begin{equation*}{m+n\choose j}-{m+n\choose j-1}-{m\choose j}+{m\choose j+q} -{n\choose j}+{n\choose j+p}\geq 0.\end{equation*}
\end{corollary}
\begin{proof}
If ${m+n\choose j}-{m+n\choose j-1}-{m\choose j} +{m\choose j+q}>0$, there are two possibilities.  In the first case, suppose ${m+n\choose j}-{m+n\choose j-1}-{m\choose j}>0$.  Apply Proposition~\ref{prop_greater_zero} for the desired result:
\begin{equation*}
\resizebox{.98\hsize}{!}{${m+n\choose j}-{m+n\choose j-1}-{m\choose j}+{m\choose j+q} -{n\choose j}+{n\choose j+p} \geq {m+n\choose j}-{m+n\choose j-1}-{m\choose j} -{n\choose j} \geq 0$}
\end{equation*}

Otherwise, suppose ${m+n\choose j} - {m+n\choose j-1} -{m\choose j} \leq 0$.  We apply Proposition~\ref{prop_less_zero} to see the quantity $-{n\choose j}+{n\choose j+p}$ is either negative one or zero.  Applied to the initial assumption, \\${m+n\choose j}-{m+n\choose j-1}-{m\choose j} +{m\choose j+q}>0$, we have
\begin{equation*}{m+n\choose j}-{m+n\choose j-1}-{m\choose j} +{m\choose j+q}-{n\choose j}+{n\choose j+p}\geq0.\end{equation*}
\end{proof}

Now, we are ready to prove rank unimodality.

\begin{proposition}\label{prop_unimodular}
For integers $m,n,p,q > 0$, the rank sequence $(r_i)_{i\geq 1}$ of the poset $C(B_{m,p}\cdot B_{n,q},r)$ of connected, vertex-induced subtrees containing the root is unimodal if $q< m$ or $p<n$.
\end{proposition}

\begin{proof}
As long as either $q< m$ or $p<n$, the sequence of first differences $(d_i)_{i\geq 2}$ is non-negative up to $i=p+q$.  Using the change of variables in Eq.~\ref{eq_jsimplified}, we need to check that the signs of the first differences $d_{q+p+j+1}$ are positive and then possibly negative for $0\leq j \leq \lfloor \frac{m+n}{2} \rfloor$, as we know they are negative for $j> \lfloor \frac{m+n}{2} \rfloor$.  Without loss of generality, suppose $m\geq n$.  Because for $j>1$, Corollaries~\ref{cor_less_zero} and~\ref{cor_greater_zero} show ${m+n\choose j}-{m+n\choose j-1}-{m\choose j}+{m\choose j+q} -{n\choose j} +{n\choose j+p}$ is non-positive or non-negative, respectively, when ${m+n\choose j}-{m+n\choose j-1}-{m\choose j}+{m\choose j+q}$ is non-positive or positive, respectively, we see their signs are positive and negative or possible neutral on the same intervals.  As the sequence ${m+n\choose j}-{m+n\choose j-1}-{m\choose j}+{m\choose j+q}$ is unimodal on the interval in question by Proposition~\ref{prop_c_unimodular} and for $j=2$ the difference $d_{q+p+3} \geq mn-m-n \geq 0$, the first difference of the rank sequences changes sign at most once.
To complete the proof, we need to consider when $j=1$.  As long as either $q<m$ or $p<n$, we note the first difference $d_{q+p+2} \geq 0$ as seen in the proof of Proposition~\ref{prop_notunimodular}.  Further we check that
\begin{equation*}
d_{q+p+1} > {m\choose q}+ {n\choose p}-1 \geq n-1 \geq 0
\end{equation*}
so the difference in rank sequence is positive for $j=0,1,2$ as long as $q<m$ or $p<n$ and $m,n \geq 3$ or $m>n\geq 2$. Thus the sequence of first differences is positive and then negative and therefore the rank sequence is unimodal.
\end{proof}

As we have already discussed the unimodality of the rank sequence $(r_i)_{i\geq 1}$ of the poset $C(B_{m,p} \cdot B_{n,q},r)$ where $p=q=0$ and $q>p=0$, we still have two special cases to consider.

\begin{lemma}\label{lemma_cases}
The rank sequences $(r_i)_{i\geq 1}$ of the posets  $C(B_{m,p} \cdot B_{1,q},r)$ and $C(B_{2,p}\cdot B_{2,q},r)$ are unimodal for any integers $m,p,q> 0$.
\end{lemma}

\begin{proof}
Let $n=1$.  Because $p>0$, we have $p\geq n$, so suppose $m>q>0$. In this case, the first difference $d_i$ is non-negative up through $i=p+q+2$.  If $i>p+q+2$, then
\begin{equation*}
d_i = d_{p+q+j+1} = {m \choose j+q}-{m\choose j-2}
\end{equation*}
for $j>1$.  If $q<m-1$, the difference $d_{q+p+3} \geq 0$ and hence Proposition~\ref{prop_unimodular} implies rank unimodularity of $C(B_{m,p} \cdot B_{1,q},r)$.  If $q=m-1$, we have $d_{q+p+j+1}\leq 0$ for all $j>1$ and hence the difference changes sign exactly once at $j=2$, implying the poset is also rank unimodular.  Now check in the case $q\geq m$.  By Corollary~\ref{cor_difference}, the first difference $d_i$ is positive for $i=2$ up to $i= \max \{p,q\}$.  Then, for $\max \{p,q\} < i \leq q+p$, the first difference \begin{equation*}d_i= {1\choose i-q-1} -1 + {m\choose i-p-1}\end{equation*} is non-negative if and only if $i\leq \max\{q+2, m+p+1\}$, so thus the sign of the first difference changes from positive to negative when $i=\max \{q+2, m+p+1\} +1$. For $i> q+p$, equation~\ref{eq_jsimplified} gives the first difference
\begin{equation*}
d_i = d_{p+q+j+1} = {1\choose j+p} -{1\choose j} + {m\choose j+q} - {m\choose j-2} 
\end{equation*}
which is always less than or equal to zero for $j\geq 0$.  Hence the sequence of differences is positive and then negative so the rank sequence of $C(B_{m,p} \cdot B_{1,q},r)$ is unimodal.  

Next, without loss of generality, assume $p\leq q$ and consider the poset $C(B_{2,p}\cdot B_{2,q},r)$.  The rank sequence of $C(B_{2,p}\cdot B_{2,q},r)$ is the convolution of two sequences $(t_i)_{i\geq 1}$ and $(s_i)_{i\geq1}$ consisting of exactly one two with the remaining terms being ones.  First consider the convolution of the sequence of $(p+3)$ ones with the sequence of $(q+3)$ ones.  At each rank, either we add one, remain neutral, or subtract one from the value of the previous rank, so the sequence of differences is 
\begin{equation*}
(d_i)_{i\geq 2} = 1,1,\cdots, 1, 0, 0, \cdots, 0, -1, -1, \cdots, -1,
\end{equation*}
where there are $(p+2)$ ones, $(q-p)$ zeros, and $(p+2)$ negative ones.  Replacing the $(p+2)^{nd}$ one in the first sequence with a two, that is setting $t_{p+2}=2$, increases the term $d_{p+2}$ in the sequence of differences by one and replacing the $(q+2)^{nd}$ term in the second sequence with a two, that is setting $s_{q+2}=2$, increases the difference $d_{q+2}$ by one.  These terms are either ones or zeros so increasing by one (or possibly two if $p=q$) does not change the sign of the difference sequence from positive to negative.  Further, the only other change to the difference sequence is in the last three terms of the sequence which become $0, -2, -3$.  These terms do not change the sign on the difference sequence from negative to positive, so the difference sequence changes sign exactly one and hence the rank sequence is unimodal.
\end{proof}

Thus, Propositions~\ref{prop_notunimodular} and~\ref{prop_unimodular} with Corollary~\ref{cor_cases} and Lemma~\ref{lemma_cases} provide necessary and sufficient conditions on the integers $m,n,p,q$ to determine the rank unimodality of the subtree poset $C(B_{m,p}\cdot C_{n,q},r)$ and prove Theorem~\ref{thm_seq_conditions}.  In the final section, we enumerate the trees that correspond to a poset $C(B_{m,k}\cdot B_{n,p},r)$ which is not rank unimodal.

\section{Enumeration}\label{sec_enumeration}
As we shall see, there exists an explicit formula to count the number of trees $B_{m,p} \cdot B_{n,q}$ whose posets $C(B_{m,p} \cdot B_{n,q}, r)$ are not rank unimodal.  But first, we will utilize a bijection between two sets.

\begin{proposition}\label{prop_bijection}
For an integer $i\geq 1$, the set of $(2\times i)$ binary matrices with no zero rows or columns, up to row and column permutation is in bijection with the set of integer pairs $(m,n)$ such that $m\geq n\geq 3$ or $m>n=2$ and $m+n \leq i+4$.

\end{proposition}
\begin{proof}
We provide a bijective map from the set $B_i$ of integer pairs $(m,n)$ such that $m\geq n \geq 3$ or $m>n=2$ and $m+n\leq i+4$ into the set $A_i$ of $(2\times i)$ binary matrices with no zero rows or columns, up to row and column permutation.
Given an integer $i\geq 1$ and a pair $(m,n)$ from the set $B_i$, suppose $m\geq n>2$.  Set $(a,b)=(m-2,n-2)$.  Then the image of $(m,n)$ is the $2\times i$ matrix with $b$ zeros in the first row and $a$ zeros in the second row as follows:
\[
\begin{bmatrix}
1 & \cdots & 1& 0 & \cdots &0 &1 &\cdots &1\\
0 & \sunderb{2.5em}{a} &0 & 1 & \sunderb{2.5em}{b} &1 &1 &\sunderb{2.5em}{i-a-b} & 1\\
\end{bmatrix}
\]
When $n=2$, set $(a,b)=(m-3,0)$ and map $(m,2)$ to the matrix
\[
\begin{bmatrix}
 1 & \cdots &1 &1 &\cdots &1\\
 0 & \sunderb{2.5em}{a} &0 &1 &\sunderb{2.5em}{i-a} & 1\\
\end{bmatrix}.
\]
This map is injective; for each unique integer pair $(m,n)$ the pair $(a,b)$ is also unique and hence describes a unique matrix up to row and column permutation.  Further, given a $(2\times i)$ matrix from the set $B_i$, the map may be reversed by using row and column permutations to standardize the matrix so the first row is a sequence of ones followed by a sequence of zeros followed by a sequence of ones, and the second row is a sequence of zeros followed by a sequence of ones.  Then the pairs $(a,b)$ are easily determined and can be shifted to pairs $(m,n)$ in $B_i$.
\end{proof}

Set $a_i = |A_i|$.  The sequence $(a_i)_{i\geq 0}$ is a known sequence appearing (with a shift) as sequence \seqnum{A024206} in the OEIS~\cite{OEIS}.  One formula is 
\begin{equation*}
a_{i}  = \left \lfloor \frac{i(i+4)}{4} \right\rfloor.
\end{equation*}

Now, for a given number of vertices in the tree, we will enumerate the trees whose subtree poset $C(B_{m,p}\cdot B_{n,q},r)$ is not rank unimodal.

\begin{theorem}\label{thm_count}
Let $T=B_{m,p}\cdot B_{n,q}$ be a tree which is the merge of two broom graphs for some integers $m,n>0$ and $p,q\geq 0$.  Let $b_i$ be the number of trees $T$ with $i$ vertices such that the poset $C(T,r)$ is not rank unimodal.  We have $b_1=b_2 = \cdots = b_9 =0$, and for $k\geq 0$
\begin{align*}
b_{2k+10} &= 2 \left( \sum_{i=0}^k \left\lfloor \frac{i(i+4)}{4} \right\rfloor \right)  - \left\lfloor \frac{k^2}{4}\right\rfloor\\
b_{2k+11} &= 2 \left(\sum_{i=0}^k \left\lfloor \frac{i(i+4)}{4} \right\rfloor \right)+ \left\lfloor \frac{(k+1)(k+5)}{4} \right\rfloor  - \left\lfloor \frac{k^2}{4}\right\rfloor
\end{align*}
\end{theorem}

\begin{proof}
Every tree with $i$ vertices that is the merge of two brooms may described by an ordered quadruplet $(m,n,p,q)$ of non-negative integers where $m$ gives the number of pendant vertices on the first broom, $n$ gives the number of pendant vertices on the second broom, $p$ describes the length of the handle of the first broom, and $q$ describes the length of the handle of the second broom.  The sum is $m+n+p+q = i-1$.  Without loss of generality we will always assume $m \geq n$.

If the rank sequence of the poset $C(B_{m,p} \cdot B_{n,q})$ is not unimodal, by Theorem~\ref{thm_seq_conditions}, we know $m\geq 3$, $n\geq 2$, $q \geq m$ and $p \geq n$.  First fix the first two terms of the quadruplet $(m,n)$.  Under these restrictions, for each pair $m+n \leq \frac{i-1}{2}$, the number of trees leading to a non-unimodal poset is completely determined, that is, we have the following set of quadruplets describing such trees:
\begin{center}
$\{(m,n,m,i-2m-n-1), (m,n,m+1, i-2m-n-2), 
\ldots, (m,n,i-m-2n-1,n)\}$.
\end{center}
As long as $m\not=n$ each of these trees is unique as a rooted tree up to isomorphism.  If $m=n$, we have the following unique trees:
\begin{center}
$\{(m,n,m,i-2m-n-1), (m,n,m+1, i-2m-n-2), 
\ldots, (m,n, \lfloor \frac{i-m-n-1}{2} \rfloor,\lceil \frac{i-m-n-1}{2} \rceil ) \}.$
\end{center}
We work recursively.  If $i$ is odd, for every tree $(m,n,p,q)$ on $i-1$ vertices there is a corresponding tree on $i$ vertices using the straight-forward injection
\begin{equation*}
(m,n,p,q) \xhookrightarrow{}  (m,n,p,q+1).
\end{equation*} 
Further for each pair $(m,n)$ such that $m+n\leq \frac{i-1}{2}$ there is one additional tree to be counted, that is, the tree with $n$ as the final coordinate, $(m,n, i-m-2n-1,n)$ if $m\not= n$, and the tree $\left(m,m,\frac{i-2m-1}{2}, \frac{i-2m-1}{2}\right)$  if $m=n$.  None of these trees are accounted for in the surjection as their last coordinates are smaller than those given in the map.  
Thus to count the number of trees when $i$ is odd, we may count the number of trees on $i-1$ vertices plus the number of pairs $(m,n)$ such that $m+n\leq \frac{i-1}{2}$.  Setting $i=2k+11$, recall from Proposition~\ref{prop_bijection} the number of such pairs are given by the sequence $(a_{k+1})$.  Therefore for $k\geq 0$,
\begin{equation*}
b_{2k+11} = b_{2k+10} + a_{k+1}.
\end{equation*}
Now suppose $i$ is even.  We still have the injection from the set of trees on $i-1$ vertices to the set of trees on $i$ vertices, however we must be careful when counting additional trees.  If $m=n$ there is no additional tree as the tree $(m,m, \left \lfloor \frac{i-2m-1}{2} \right \rfloor+1, \left \lceil \frac{i-2m-1}{2}\right \rceil-1)$ is isomorphic to the tree $(m,m,\left \lfloor\frac{i-2m-1}{2}\right \rfloor , \left \lceil \frac{i-2m-1}{2}\right \rfloor)$.

Therefore for $i$ even we add the number of trees on $i-1$ vertices plus the number of pairs $(m,n)$ where $m\not = n$.  Setting $i=2k+10$, the number of pairs $(m,n)$ where $m+n\leq \frac{i-1}{2}$ is given by ${a_k}$, and the number of such pairs where $m=n$ is $\lfloor \frac{k}{2}\rfloor$.  Now, we have 
\begin{equation*}
b_{2k} = b_{2k-1} + a _{k} - \left\lfloor \frac{k}{2} \right\rfloor.
\end{equation*}
Now, we repeatedly apply the recursion along with the initial conditions $b_{10}=0$ and $b_{11} =1$ found in by Jacobson, K\'ezdy, and Seif~\cite{Jacobson_Kezdy_Seif}.
\begin{align*}
b_{2k+11} &= b_{2k+10} + a_{k+1}\\
&= b_{2k+9} +a_{k}+ a_{k+1} -\left\lfloor \frac{k}{2}\right\rfloor\\
&= b_{2k+8} +2a_{k}+ a_{k+1} -\left\lfloor \frac{k}{2}\right\rfloor\\
&= b_{2k+7} +a_{k-1}+ 2a_{k}+ a_{k+1} -\left\lfloor \frac{k}{2}\right\rfloor - \left\lfloor \frac{k-1}{2}\right\rfloor\\
&\vdots \\
&= 2(a_1+ a_2 + \cdots a_{k}) + a_{k+1} - \left\lfloor \frac{k}{2}\right\rfloor - \cdots - \left\lfloor \frac{1}{2} \right\rfloor\\
&= 2 \left(\sum_{i=1}^k \left\lfloor \frac{i(i+4)}{4} \right\rfloor \right)+ \left\lfloor \frac{(k+1)(k+5)}{4} \right\rfloor  - \left\lfloor \frac{k^2}{4} \right \rfloor
\end{align*}
Note the identity $\sum_{i=1}^{k} \left\lfloor \frac{i}{2} \right\rfloor = \left \lfloor \frac{k^2}{4} \right\rfloor$ may be found in the entry for sequence \seqnum{A002620} of the OEIS~\cite{OEIS}.

This result on the odd indexed terms then implies for $k\geq 0$,
\begin{equation*}
b_{2k+10}=2 \left(\sum_{i=1}^k \left\lfloor \frac{i(i+4)}{4} \right\rfloor \right)  - \left\lfloor \frac{k^2}{4} \right \rfloor.
\end{equation*}
\end{proof}

Trees leading to a vertex-induced poset that is not unimodal are quite frequent.  Let $t_i$ be the total number of non-isomorphic trees $B_{m,p} \cdot B_{n,q}$ with $i$ vertices, $m,n\geq 0$ and $p,q >0$.  The sequence $(t_i)_{i\geq 1}$ is found in \textit{The Online Encyclopedia of Integer Sequences}~\cite{OEIS} as sequence \seqnum{A005993} and $(b_i)_{i\geq 1}$ is found in sequence \seqnum{A320657}.  Table~\ref{tab_notunimodular} gives the first values of the sequence $(b_i)$ starting with $i=10$, that is the number of trees of a given size whose poset $C(B_{m,p} \cdot B_{n,q},r)$ has a rank sequence which is not unimodal, determined by Theorem~\ref{thm_count}, as well as the total number of non-isomorphic trees $B_{m,p}\cdot B_{n,q}$ with $i$ vertices where $m,n>0$ and $p,q \geq 0$.
\begin{table}
\caption{Initial values for sequences $(b_i)$ and $(t_i)$, respectively, which enumerate the number of non-unimodal, total number respectively, of sequences $(r_i)_{i\geq 0}$ where $m,n,\geq 0$ and $p,q>0$.}
\label{tab_notunimodular}
\begin{center}
\begin{tabular}{c|ccccccccccccc}
\hline\hline\noalign{\smallskip}
$i$  & 10& 11 & 12 & 13 & 14 & 15 & 16 & 17 & 18 & 19 & 20 & 21 & 22\\
\noalign{\smallskip}\hline\noalign{\smallskip}
$b_i$  & 0& 1 & 2 & 5 & 7 & 12 & 16 & 24 & 30 & 41 & 50 & 65 & 77\\
\noalign{\smallskip}\hline\noalign{\smallskip}
$t_i$ & 60 &85 &110 & 146 & 182 & 231 & 280 &344 & 408 &570 &670 & 770 &891\\
\noalign{\smallskip}\hline
\end{tabular}
\end{center}
\end{table}

%
%
%
%
%

\end{document}